\documentclass{amsart}
\usepackage{newtxmath}
\usepackage{tikz-cd}
\usepackage{anyfontsize}

\graphicspath{{figures/}}

\newtheorem{thm}{Theorem}[section]
\newtheorem{defn}[thm]{Definition}
\newtheorem{prop}[thm]{Proposition}
\newtheorem{lem}[thm]{Lemma}
\newtheorem{cor}[thm]{Corollary}

\DeclareMathOperator{\arcch}{arccosh}

\DeclareMathOperator{\image}{Im}

\begin{document}

\title[Weighted Voronoi-Delaunay dual]{Weighted Voronoi-Delaunay dual on polyhedral surfaces and its finiteness}
\author{Xiang Zhu}
\address{Department of mathematics, Shanghai University, Shanghai, China, 200444.}
\email{zhux@shu.edu.cn}

\begin{abstract}
    We aim to give a strict proof of the existence and uniqueness of the weighted Voronoi decomposition and the dual weighted Delaunay triangulation on Euclidean and hyperbolic polyhedral surface as well as hyperbolic surface with geodesic boundaries. Since the former definition of the Voronoi cell may not be simply connected, we slightly adjust the definition. Our proof is to construct an isotopic map instead of using the edge-flipping algorithm, which is a generalization of the one by Dyer et al. The main theorem of this paper is a lemma for proving the existence of the inversive distance circle packing.
\end{abstract}
\maketitle

\section{Introduction}

\subsection{Statement of results}
Voronoi diagram and Delaunay triangulation are well known in computational geometry. Roughly speaking, given finite points in $\mathbb{R}^2$, the Voronoi cell of a point is the collection of all points in $\mathbb{R}^2$ that are closer to the selected point than to any other points. These Voronoi cells compose a graph called Voronoi diagram, of which the dual graph is Delaunay triangulation when the points are in the general position. The duality of these two graphs is called Voronoi-Delaunay duality.

One natural generalization of Voronoi-Delaunay duality is weighted Voronoi-Delaunay duality. Given each point a real number called weight, we measure the point in $\mathbb{R}^2$ between the selected ones using the distance squared minus weight instead of the distance itself, then we have weighted Voronoi cells. The weighted Voronoi cells of this generalization are convex polygons if they exist. 

Another generalization is to study the duality on the piecewise flat surface, or also known as polyhedral surface. Here all the cone points of the polyhedral surface should be contained in the selected points above, namely the vertices of the triangulation. 

Our work is to generalize weighted Voronoi-Delaunay duality on three kinds of surface.
\begin{itemize}
    \item Piecewise flat surface
    \item Hyperbolic polyhedral surface
    \item Hyperbolic surface with geodesic boundaries
\end{itemize}
We slightly adjust the definition of Voronoi cell to make it simply connected. In the third case, we find that comparing the weight times hyperbolic sine of distance is a sound way to define weighted Voronoi cell with good propositions.

We prove the existence and uniqueness of these dualities. Moreover, we prove that when the weights vary, the number of weight Delaunay triangulations is finite.

\subsection{Related work}

Rivin proposed and proved the existence and uniqueness of Voronoi-Delaunay duality on piecewise flat surfaces in \cite{rivin1994euclidean}. However, Bobenko and Springborn in \cite{bobenko2007discrete} pointed out that Rivin's proof is flawed, while Indermitte's proof is almost correct \cite{indermitte2001voronoi}, and gives another strict proof themselves. In fact, the main ideas of the three proofs above are the same. Starting from any triangulation, when 2-2 edge flipping is performed on any edge that do not meet Delaunay condition, each step will lead to a strict decrease of a positive function with respective to the triangulations. If there are finite triangulations such that the function is bounded, then the flip will be terminated in finite steps, then the Delaunay triangulation is obtained. The flaw of Rivin's proof is that when the selected function is bounded, there may be infinite triangulations such that the function is below the bound.

To the best of our knowledge, on weighted Voronoi-Delaunay duality on piecewise-flat surfaces, the latest detail research before the author's thesis \cite{zhu2019discrete} would be Gorlina's thesis \cite{gorlina2011weighted}. In this paper, the edge flipping method is used to prove the existence and uniqueness of weighted Delaunay triangulation on piecewise-flat surface. However, Gorlina's proof cited a Dirichlet energy \cite{glickenstein2007monotonicity} defined by Glickenstein. Similar to the function used by Rivin, when the energy is bounded, the number of triangulations might be infinite as well. Thus, Gorlina's proof is also flawed. Recently, Lutz uses a hyperbolic volume as an energy function to fix Gorlina's proof by flipping and generalize the theorem into mixed types on hyperbolic polyhedral surfaces.\cite{lutz2022Canonical}

In this paper, although we can define a Euclidean volume function to follow the proofs by edge-flipping, instead, we study the weighted Voronoi decomposition directly by definition, and construct a cover of isotopy from its dual graph to a graph with all edges geodesic, then the weighted Delaunay triangulation is obtained. This idea follows the paper \cite{dyer2007voronoi} by Dyer et al., however, a more strict description of definitions and proofs is presented.
This isotopy method is the translation of part of content in \cite{zhu2019discrete} from Chinese into English.

From Akiyoshi's paper \cite{akiyoshi2001finiteness}, we know that given a hyperbolic surface with cusp ends, if we decorate a horocycle on every cusp and change the size of these horocycles, the number of induced weighted Delaunay triangulations of the decorated hyperbolic surface is finite. In this paper, we will generalize Akiyoshi's theorem to the three kinds of surface mentioned above.

\subsection{Organization of the paper}

In section 2 we recall isotopy and $\Delta$-decomposition to define the triangulation in our paper. Then we recall the piecewise flat surface and the hyperbolic polyhedral surface.

In section 3 we prove the existence and uniqueness of weighted Voronoi decomposition on piecewise flat surface. With the adjustment the definition, the weighted Voronoi decomposition becomes CW.

In section 4 we prove the existence and uniqueness of weighted Delaunay triangulation on piecewise flat surface by constructing an isotopy map directly.

In section 5 and 6 we do the same thing in section 3 and 4 on hyperbolic polyhedral surface and hyperbolic surface with geodesic boundaries respectively.

In section 7 we prove the finiteness of weight Delaunay triangulation on these three kinds of surfaces when the weights vary.


\section{Preliminary}

\subsection{Isotopy}
To describe the triangulation of surface and the dual of complex precisely, we might review the definition of \emph{isotopy} \cite{rourke2012introduction}.
\begin{defn}	
    Suppose $X,Y$ are topological space, $I=[0,1]$.
    \begin{itemize}         
        \item A map is \emph{level-preserving} if $F \colon X \times I \to Y \times I$ for each $t \in I$. We can then define $F_t \colon X \to Y$ by $F(x,t)=(F_t(x),t)$.
        \item An \emph{isotopy of $Y$} is a level-preserving homeomorphism $H \colon Y \times I \to Y \times I$ such that $H_0=id$. We say that $H_1$ is \emph{isotopic to the identity}.
        \item An \emph{isotopy of $X$ in $Y$} is a level-preserving embedding $F \colon X \times I \to Y \times I$, and we say that the embeddings $F_0$ and $F_1$ are \emph{isotopic}.
        \item We say that $H$ \emph{covers} $F$ if $F=H \circ (F_0 \times id)$
        \item An isotopy \emph{fixes} a subset $V \subset X$ if $F|_{V \times I} =(F_0 \times id)|_{V \times I}$. 
    \end{itemize}
\end{defn}
We know that isotopy is an equivalence relation and implies homotopy.

\subsection{Triangulation}

Suppose $S$ is a connected two-dimensional topological surface. In this paper, we call a $\Delta$-complex decomposition up to isotopy fixing 0-cells a \emph{(topological) triangulation} of $S$. This definition differs from the commonly used one defined by simplicial complex decomposition, and allows the vertices and edges of one single triangle to glue itself. The definition of $\Delta$-decomposition is below \cite{hatcher2002algebraic}.

\begin{defn}
    Suppose $X$ is a non-empty topological space, and $\Delta^n$ is an $n$-dimension simplex, if there is a collection of maps $\sigma_\alpha \colon \Delta^n \to X,\,\alpha \in A$  with $n$ depending on the index $\alpha$, such that
    \begin{itemize}
        \item $X = \bigsqcup_{\alpha \in A} \image \sigma_\alpha$ is a CW decomposition and $\sigma_\alpha$ is every characteristic map.
        \item For any $\alpha \in A$, there exists $\beta \in A$ such that $\sigma_\beta$ is the restriction of $\sigma_\alpha$ on the face of $\Delta^n$.
    \end{itemize}
    Then we call the CW decomposition a \emph{$\Delta$-decomposition}.
\end{defn}

\subsection{Piecewise flat metric}

This part is from \cite{troyanov1986surfaces}. We use the polar coordinate $(r,\theta)$ on the Euclidean plane with the metric of $dr^2+r^2d\theta^2$, then define a \emph{Euclidean cone} by the set
\begin{equation}
    \left\{\, (r,\theta) \mid r \ge 0, \, \theta \in \mathbb{R}/\varphi\mathbb{Z}
    \,\right\} / (0,\theta_1) \sim (0,\theta_2)
\end{equation}
with the metric of $dr^2+r^2d\theta^2$ and the cone angle of $\varphi$.
The origin of the set becomes the cone point of the cone.

Then we give a metric to a marked surface $(S,V)$, which is an oriented closed surface  $S$ with a non-empty finite subset $V \subset S$. The elements of $V$ is called \emph{vertices}.

For any point $p\in S\setminus V$, there is a neighborhood isometric to a region in $\mathbb{E}^2$. For any vertex $v_i \in V$ there is a neighborhood isometric to a region of a Euclidean cone, of which the cone point coincides with the vertex and the cone angle is denoted by $\varphi_i$. We call the metric of the surface a \emph{polyhedral metric} or \emph{piecewise flat metric}, denoted by $d_f$, and the subscript $f$ indicates flat. Meanwhile, we use $d_f \colon S \times S \to \mathbb{R}_{\ge 0}$ to indicate the distance function on the surface, and $(S,V,d_f)$ is called \emph{piecewise flat surface} or \emph{polyhedral surface}.

The following theorems describe the geodesic and the triangulation on piecewise flat surface \cite{mitchell1987the}. 
\begin{thm}\label{euclidean geodesic}		
    Any geodesic on the piecewise flat surface $(S,V,d_f)$ is a folded straight line segment. If it passes through a vertex, then the angle made by the geodesic path at the vertex is no less than $\pi$.

    For any two points $p,q \in S$, there is at least one geodesic path connecting $p,q$, with the distance of $d_f(p,q)$. This geodesic path has no self-intersection.
\end{thm}

\begin{thm}\label{triangle}		
    There exist a \emph{geodesic triangulation} on the piecewise flat surface $(S,V,d_f)$, that is, $\Delta$- decomposition with $V$ as the vertex and all edges as geodesics.
\end{thm}

\subsection{Hyperbolic polyhedral metric}

This part is from \cite{gu2018discreteII}. As it is very similar to the previous part, we only need to repeat most conclusions here. We use the polar coordinate $(r,\theta)$ on Poincaré disk model of the hyperbolic plane $\mathbb{H}^2$, that is, the unit disk $\mathbb{D}$ with the metric of $\frac{4(dr^2+r^2d\theta^2)}{(1-r^2)^2}$ and the Gaussian curvature of $-1$. A geodesic of $\mathbb{H}^2$ is an arc perpendicular to $\partial \mathbb{D}$.

We call the set
\begin{equation}
    \left\{\, (r,\theta) \mid 0 \le r < 1, \theta \in \mathbb{R}/\varphi\mathbb{Z} \,\right\} / (0,\theta_1) \sim (0,\theta_2)
\end{equation}
with the metric of $\frac{4(dr^2+r^2d\theta^2)}{(1-r^2)^2}$ a \emph{hyperbolic cone} with the cone angle of $\varphi$.

Then we give another metric to the marked surface $(S,V)$.
For any point $p\in S\setminus V$ there exists a neighborhood isometric to some area in $\mathbb{H}^2$. For any vertex $v_i \in V$ there exists a neighborhood isometric to a region of a hyperbolic cone, of which the cone
point coincides with the vertex and the cone angle is denoted by $\varphi_i$.

In this case, the metric given to this surface is called \emph{hyperbolic polyhedral metric}, denoted by $d_h$, and the subscript $h$ indicates hyperbolic. Meanwhile, we also regard $d_h \colon S \times S \to \mathbb{R}_{\ge 0}$ as the distance function on the surface. We call $(S,V,d_h)$ a \emph{Piecewise surface with constant negative curvature} or \emph{hyperbolic polyhedral surface}. Similarly, the discrete curvature at the vertex $v_i$ is defined by $K_i = 2\pi-\varphi_i < 2\pi$.

Similar to the theorem \ref{euclidean geodesic}, from the simple knowledge of Riemannian manifold, the following theorems hold.
\begin{thm}
    Any geodesic path on the hyperbolic polyhedral surface $(S,V,d_h)$ is a folded hyperbolic geodesic line segment. If it passes through a vertex, then the angle made by the geodesic path at the vertex is no less than $\pi$.
        
    For any two points $p,q \in S$, there is at least one geodesic path connecting $p,q$, with the distance of $d_h(p,q)$. This geodesic path has no self-intersection.
\end{thm}

\begin{thm}\label{hyp_trig}		
    There exists a geodesic triangulation on the hyperbolic polyhedral surface $(S,V,d_h)$.
\end{thm}

The theorem above is a trivial generalization of the theorem \ref{triangle}.
In particular, the Delaunay triangulation on $(S,V,d_h)$ is obviously a geodesic triangulation, which can be regarded as a proof \cite{gu2018discreteII}. 
Moreover, we will prove the existence of weighted Delaunay triangulation later, which are also geodesic triangulation of hyperbolic polyhedral surface.

\section{Weighted Voronoi decomposition}

Given a piecewise flat surface $(S,V,d_f)$, we define the \emph{weight function} on every vertex to be $\mathbf{w} \colon V \to \mathbb{R},v_i \mapsto w_i$, or $\mathbf{w} \in \mathbb{R}^V$. The original definition of \emph{weighted Voronoi cell} is
\begin{equation}
    \left\{\,p \in S \mid d_f(p,v_i)^2-w_i \le d_f(p,v_j)^2-w_j,
    \forall j \ne i\,\right\}.
\end{equation}

However, in our setting we need some restriction: the weight of any vertex should be positive, and lower than the square of the distance between any other vertices, which is
\begin{equation}
    W = \{\,\mathbf{w} \in \mathbb{R}^V \mid 
    0<w_i<d_f(v_i,v_j)^2+w_j,\forall j \ne i \,\}.
\end{equation}
This is a neat sufficient condition to ensure that every weighted Voronoi cell is not empty.

When we consider the case of inversive distance circle packing, the circles are disjoint and the radii are the square root of weights, so we have the weight function in another set
\begin{equation}
    W' = \left\{\,\mathbf{w} \in \mathbb{R}^V \mid 
    0<w_i<\mathrm{Inj}(v_i)^2 \mbox{ and }\forall i \ne j,\,
    \sqrt{w_i}+\sqrt{w_j}<d_f(v_i,v_j) \,\right\}.
\end{equation}
Here $\mathrm{Inj}$ indicates the injective radius at some point on the piecewise flat metric. Obviously, since $W' \subset W$, we only need to prove all the theorems in this paper when $\mathbf{w} \in W$, then the theorems apparently hold for $\mathbf{w} \in W'$.

\begin{defn}	
    Given a piecewise flat surface $(S,V,d_f)$, and a weight function $\mathbf{w} \in W$, the \emph{inner weighted Voronoi cell} of $v_i$, denoted by $Vor_f(v_i)$, is defined to be the set of all $p\in S$ satisfying that 
    \begin{itemize}
        \item there exists a unique shortest geodesic on $S$ connecting $p$ and $v_i$, and
        \item for any $j\neq i$, $d_f(p,v_i)^2-w_i < d_f(p,v_j)^2-w_j$.
    \end{itemize}    
\end{defn}

We know by definition that our inner weighted Voronoi cell of $v_i$ is equal to the interior of the original weighted Voronoi cell of $v_i$ minus the cut locus with respect to $v_i$ in $S$.

In this section we aim to prove that:
\begin{thm}\label{vor_f}		
    There exists a unique CW decomposition of $S$, of which the set of all the 2-cells is $\{\, Vor_f(v_i) \mid v_i \in V \,\}$.  
\end{thm}
We call this decomposition the \emph{weighted Voronoi decomposition}.  To prove this, we need the following propositions in steps.

\begin{prop}\label{vertex_vor_f}
    We have $v_i \in Vor_f(v_i)$, and for any $j \ne i$, we have $v_j \notin Vor_f(v_i)$.
\end{prop}

\begin{proof}	
    For any $j \ne i$, since $w_j<d_f(v_i,v_j)^2+w_i$, we know that $d_f(v_i,v_i)^2 - w_i = 0 - w_i < d_f(v_i,v_j)^2 - w_j$, then $v_i \in Vor_f(v_i)$.

    Similarly, we know that $d_f(v_j,v_i)^2 - w_i > -w_j = d_f(v_j,v_j)^2 - w_j$, then $v_j \notin Vor_f(v_i)$.    
\end{proof}

\begin{prop}\label{open}		
    $Vor_f(v_i)$ is an open set.		
\end{prop}

\begin{proof}		
    For any $p \in Vor_f(v_i)$, let $\gamma$ be the shortest geodesic connecting $p$ and $v_i$, and $\gamma'$ be the second shortest geodesic connecting $p$ and $v_i$, then $d_f(p,v_i)=length(\gamma)<length(\gamma')$. If there is no qualified $\gamma'$, let $length(\gamma') = + \infty$.
    
    From the continuity of distance function and other inequalities of $d_f(p,v_i)<d_f(p,v_j)$ that need to be satisfied, there exists a neighborhood of $p$, in which all the points satisfy the above inequalities, so $p$ is the interior point of $Vor_f(v_i)$, and $Vor_f(v_i)$ is an open set.
\end{proof}

\begin{prop}\label{closure}	
    $\overline{Vor_f(v_i)}$ is the weighted Voronoi cell of $v_i$ and $\bigcup_{v_i \in V} \overline{Vor_f(v_i)} = S$.
\end{prop}

\begin{proof}		
    Let $Vor'(v_i)=\left\{\,p \in S \mid d_f(p,v_i)^2-w_i < d_f(p,v_j)^2-w_j,\forall j \ne i\,\right\}$, we know that $\overline{Vor'(v_i)}$ is a weighted Voronoi cell by definition, and $Vor'(v_i) \setminus Vor_f(v_i)$ is a subset of $Vor'(v_i)$. 
    
    In the subset $Vor'(v_i) \setminus Vor_f(v_i)$, all points have at least two geodesics connecting the vertex and reaching the length of $d_f(p,v_i)$ that are not isotopic on $S$ fixing $V$. Thus this subset is contained in the critical point set of the function $d_f(\cdotp,v_i)$. 
    
    As the distance function is smooth except for the vertices, using the Sard's theorem, we know that the set of critical points has measure zero on $S \setminus V$, and its inner point is an empty set. Therefore, the inner point of $Vor'(v_i) \setminus Vor_f(v_i)$ is an empty set, and their closures are the same, or $\overline{Vor_f(v_i)}=\overline{Vor'(v_i)}$.
    
    For any point $p \in S$ and any vertex $v_i \in V$, a finite number of weighted distance values $d_f(p,v_i)^2-w_i$ can be sorted, and $p$ belongs to the weighted Voronoi cell of the vertex with the shortest weighted distance.
\end{proof}

\begin{prop}\label{neighbor}		
    For any point $p \in \overline{Vor_f(v_i)}$, let the path $\gamma \colon I \to S$ be the shortest geodesic line connecting $p$ and $v_i$, and $\gamma(0)=v_i,\gamma(1)=p$, then there exists the neighborhood $U$ of $\gamma((0,1])$ isometric embedding in $\mathbb{E}^2$.    
\end{prop}	

\begin{proof}		
    From Theorem \ref{euclidean geodesic}, we know that $\gamma$ is a folded straight line segment, which may pass through vertices with negative curvature. This case should be ruled out first. Assuming that $\gamma$ passes through the vertex $v_j$, then $d_f(p,v_j) + d_f(v_j,v_i) \le length(\gamma) = d_f(p,v_i)$, therefore,
    \[
        d_f(p,v_j)^2 - w_j < d_f(p,v_j)^2 < d_f(p,v_i)^2 - d_f(v_j,v_i)^2 < d_f(p,v_i)^2 - w_i
    \]
    which is contradicted with $p \in \overline{Vor_f(v_i)}$.

    Therefore, $\gamma((0,1]) \subset S \setminus V$, which means that $\gamma((0,1])$ has a neighborhood with flat Euclidean metric, denoted by $\tilde{U} \in S$. Then $\tilde{U}$ can be immersed in $\mathbb{E}^2$. In case of that immersed $\tilde{U}$ has self intersection part, we can always cut off the redundant part and get a neighborhood of $\gamma((0,1])$ embedded in $\mathbb{E}^2$.    
\end{proof}

\begin{prop}\label{simply connected}	
    $Vor_f(v_i)$ is simply connected.
\end{prop}

\begin{proof}
    According to Proposition \ref{neighbor}, we know that all points in $Vor_f(v_i)$ are path connected with $v_i$, so $Vor_f(v_i)$ is connected.

    Assume that $Vor_f(v_i)$ is not simply connected, then it contains a non-trivial closed loop. Let $\gamma \colon I \to Vor_f(v_i),\gamma(0)=\gamma(1)$ be a non-trivial closed loop without intersection in $Vor_f(v_i)$. 
    Suppose that the maximum of $d(\cdotp,v_i)$ on $\gamma$ is taken at $\gamma(t)$ where $0<t<1$, then denote that $\gamma_1$ and $\gamma_2$ are the geodesics in the isotopic class of $\gamma([t,0])$ and $\gamma([t,1])$ separately, we have $length(\gamma_1)=length(\gamma_2)$.  
    Next we only need to prove that $\gamma_1$ and $\gamma_2$ are not isotopic in $S$ fixing $V$, so that $\gamma(t) \not\subset Vor_f(v_i)$, then prove this proposition by contradiction.

    If $\gamma_1$ and $\gamma_2$ are not homotopic on $S$, then they are not isotopic. Else if they are homotopic, then there exists a simply connected area $U$ on $S$ with $\gamma$ as the boundary. 
    Let $\alpha,\beta > 0$ be the angles at $\gamma(0)$ and $\gamma(t)$,
    and other point on the $U$ boundary is local geodesic. 
    According to Gauss-Bonnet theorem, the sum of the discrete curvature of all vertices in $U$ is $\alpha + \beta > 0$, thus there is at least one vertex in $U$, then $\gamma_1$ and $\gamma_2$ are not isotopic fixing any vertex in $U$.
\end{proof}

\begin{prop}\label{line sigment}	
    $\partial Vor_f(v_i)$ is a loop connected by finite number of straight line segments from end to end.
\end{prop}

\begin{proof}		
    Without losing the generality, suppose $p \in \partial Vor_f(v_1)$, then there exists more than two shortest geodesics connecting from $p$ to $v_i$, which are not isotopic and satisfy the identity $d_f(p,v_1)^2-w_1 = d_f(p,v_i)^2-w_i$. In some special case, $v_i=v_1$ may occurs, but the two geodesics are not isotopic.
    
    If there are exactly two such geodesics, let them be $\gamma_1$ connecting $p,v_1$ and $\gamma_2$ connecting $p,v_2$. Then, there are two open sets of $U_i \supset \gamma_i \setminus \{v_i\},i=1,2$ isometrically immersed into $\mathbb{E}^2$. Let $\gamma_{12}$ be the folded straight line connected by $\gamma_1$ and $\gamma_2$ at $p$. From Proposition \ref{neighbor}, we know that there exists an open set $U \supset \gamma_{12}\setminus \{v_1,v_2\}$ embedded into $\mathbb{E}^2$.
    The set of the points satisfies $d_f(p,v_1)^2-w_1 = d_f(p,v_2)^2-w_2$ on $\mathbb{E}^2$ is a geodesic orthogonal to the straight line connecting $v_1$ and $v_2$. Pulling it back from the isometric mapping, we get an open straight line segment passing through $p$ in $U$, so $\partial Vor_f(v_1)$ is a straight line segment near $p$.

    If there are more than three geodesics, suppose that three of them are $\gamma_i,i=1,2,3$ connecting $p,v_i$. Then, we have three open sets of $U_i \supset \gamma_i \setminus \{v_i\},i=1,2,3$ isometrically immersed into $\mathbb{E}^2$. The solution of the equations $d_f(p,v_1)^2-w_1 = d_f(p,v_2)^2-w_2 = d_f(p,v_3)^2-w_3$ on $\mathbb{E}^2$ on $\mathbb{E}^2$ is a single point in the image of $U_1$, $U_2$ and $U_3$. Pulling it back on $S$, we find that such $p$ is an isolated point since it is contained in a finite intersection of its neighborhoods. We call such $p$ a \emph{branch point}.  Note that $\partial Vor_f(v_2) \cup \partial Vor_f(v_3)$ is also some straight line segments connected at $p$, then we can conclude that the union of all the boundaries of inner Voronoi cells is locally geodesics connected together at branch points.

    Finally, we prove that the number of straight line segments around $\partial Vor_f(v_1)$ is finite. According to the two cases above, we know that for any $p \in \partial Vor_f(v_1)$, there is a neighborhood $U_p$ containing at most one branch point, and $\bigcup_{p \in \partial Vor_f(v_1)}U_p \supset \partial Vor_f(v_1)$ is an open covering. The closed surface $S$ is compact with the metric $d_f$, then the bounded closed set $\partial Vor_f(v_1)$ is also compact. So there exists a finite open covering, and each open cover has at most one branch point. Thus, the number of branch points are finite.
\end{proof}

From the proof above, we know that $Vor_f(v_i)$ is an intersection of several open Euclidean half planes restrict on a Euclidean cone, thus the following corollary holds.

\begin{cor}		
    $Vor_f(v_i)$ is a geodesic convex set.
\end{cor}

With all the propositions above, we can finally prove Theorem \ref{vor_f}.

\begin{proof}[Proof of theorem \ref{vor_f}]
    Let all the branch points on $\bigcup_{v_i \in V} \partial Vor_f(v_i)$ be $0$-cells, and according to Proposition \ref{line sigment}, the set left after removing these $0$-cells cells is composed of finite number of unconnected open straight line segments, and let them be $1$-cells.
    From Proposition \ref{open} and \ref{simply connected}, we can set $Vor_f(v_i)$ to be a $2$-cell.

    From Proposition \ref{closure}, we know that the disjoint union of all the cells is $S$, which is a cell decomposition. This construction is determined by $d_f$ and $\mathbf{w}$ only, thus, the cell decomposition is unique.
    From Proposition \ref{line sigment} we know that this decomposition is finite, thus it is CW.
\end{proof}

\section{Weighted Delaunay triangulation}

\begin{lem}\label{dual_isotopy}
    Given a weighted Voronoi decomposition of $(S,V,d_f)$ with weight $\mathbf{w} \in W$, let $p$ be a point on a 1-cell which is the boundary of $Vor_f(v_i)$ and $Vor_f(v_j)$. Here, $v_i$ and $v_j$ could coincide. Suppose the path $\gamma$ is the geodesic lines connecting $p,v_i$ and $p,v_j$ glued at $p$, then there exists a unique geodesic on $S$ connecting $v_i$ and $v_j$ isotopic to $\gamma$ fixing $V$.
\end{lem}

\begin{proof}		
    Suppose that the 1-cell to which $p$ belongs is $\mu \colon I \to S$. From Proposition \ref{neighbor}, for any $t \in I$, the path glued by the two geodesics connecting $v_i$, $\mu(t)$ and $v_j$ is denoted by $\gamma_t$.
    If there exists $t \in I$ such that the angle of $\gamma_t$ at $\mu(t)$ is $\pi$, then $\gamma_t$ is the geodesic isotopic to $\gamma$. Else, we only need to consider the case that the angle can not achieve $\pi$ when $t \in I$. 

    Without loss of generality, we assume that the angle on the side of $\mu(1)$ is less than $\pi$. Now we extend $\mu(t)$ in the geodesic direction of $t>1$.
    If we embed $\gamma_t$ and $\mu(t)$ and their neighborhoods into $\mathbb{E}^2$, then there exists $s>1$ such that $\gamma_s$ is the geodesic connecting $v_i$ and $v_j$ and isotopic to $\gamma_t$. To prove this lemma, we only need to show that this extension can be done on $S$ as well.

    If not, let $s'$ be the supremum of $t$ which allows this isometrically embedding hold on $S$, thus we suppose that $s' \le s$. Then we know that $\gamma_{s'}$ passes through some vertex, e.g, the geodesic connecting $v_i$ and $\mu(s')$ passes through $v_k$. Since $s' \le s$, the angle $\bigl<v_i,\mu(s'),\mu(1)\bigr>$ is no less than $\frac{\pi}{2}$. From Proposition \ref{vertex_vor_f}, we know that $v_k$, $\mu(s')$, and then $\mu(1)$ locate at the same open half space of the 1-cell with respect to $\partial Vor_f(v_i) \cap \partial Vor_f(v_k)$. Hence, the point near $\mu(1)$ cannot be in $Vor_f(v_i)$. However, $\mu(1)$ is a 0-cell on $\partial Vor_f(v_i)$, of which any neighborhood contains points of $Vor_f(v_i)$. Then we get the contradiction, and prove the existence geodesic on $S$ isotopic to any $\gamma_t$ fixing $V$.

    The uniqueness is similar to the Proposition \ref{simply connected}, which can be easily proved by using Gauss-Bonnet theorem.

\end{proof}

Now we define the weighted Delaunay decomposition on $(S,V,d_f)$ with $\mathbf{w} \in W$. Recall that the classical proof of Poincaré duality theorem
\begin{equation}
    H^k(S;\mathbb{Z}_2) \cong H_{2-k}(S;\mathbb{Z}_2) \quad k=0,1,2
\end{equation}
used the construction of the barycentric subdivision and linked sub-complex. Here we use the same way to construct duality. Let the only vertex $v_i \in Vor_f(v_i)$ be the center of $Vor_f(v_i)$, and the midpoint of 1-cells be the center of 1-cells, we barycentric subdivide and glue the linked cells around 0-cells of the weighted Voronoi decomposition, then we have a dual graph topologically. From Lemma \ref{dual_isotopy}, there exist an isotopy map fixing $V$ covering the isotopy of every edge of the dual graph to the unique geodesic. Thus, we have the following definition.

\begin{defn}\label{del_def}
    Given a piecewise flat surface $(S,V,d_f)$ with weight $\mathbf{w} \in W$, there exist a unique CW decomposition whose 1-cells are geodesics connecting vertices, called \emph{weighted Delaunay tessellation}, to be the dual graph of the weighted Voronoi decomposition.

    The \emph{weighted Delaunay triangulation} is to subdivide polygon faces of the weighted Delaunay tessellation into triangles by connecting some geodesic diagonals in any ways.
\end{defn}

There are several equivalent definitions of weighted Delaunay triangulation, and we need to verify that our definition is equivalent to one of them, e.g, Glickenstein's definition in \cite{glickenstein2008geometric}. Given a piecewise flat surface $(S,V,d_f)$ and its geodesic triangulation $\mathcal{T}=(V,E,F)$ and a weight function $\mathbf{w} \in W$, for any $f_{ijk} \in F$, if $O_{ijk} \in S$ and $w_{ijk} \in \mathbb{R}$ exist, such that the equations
\begin{equation}
    \begin{aligned}
        d_f(O_{ijk},v_i)^2-w_{ijk}&=w_i\\
        d_f(O_{ijk},v_j)^2-w_{ijk}&=w_j\\
        d_f(O_{ijk},v_k)^2-w_{ijk}&=w_k
    \end{aligned}
\end{equation}
hold, then $O_{ijk}$ and $w_{ijk}$ are unique. Moreover, if $v_l \in V$ satisfies
\begin{equation}
    d_f(O_{ijk},v_l)^2-w_{ijk} \ge w_l,
\end{equation}
then $\mathcal{T}$ is called weighted Delaunay triangulation. Obviously, our definition \ref{del_def} is the same with this definition, where $O_{ijk}$ is the ``branch points'' or 0-cells in the weighted Voronoi decomposition.

In particular, the same definition is given in the paper \cite{springborn2008variational} for the case of weight function $\mathbf{w} \in W'$, and then $w_{ijk}>0$. A circle, called \emph{orthogonal circle}, centered at $O_{ijk}$ with radius $\rho_{ijk} \coloneqq \sqrt{w_{ijk}}$ is orthogonal to three circles centered at $v_i,v_j,v_k$ with radius $\sqrt{w_i},\sqrt{w_j},\sqrt{w_k}$. 
We immerse the orthogonal circle into $S$, then any other circle centered at $v_l \in V$ with radius $\sqrt{w_l}$ does not intersect with the orthogonal circle, or intersect with the intersection angle no more than $\frac{\pi}{2}$. If this property holds for every face of the triangulation $\mathcal{T}$, then $\mathcal{T}$ is weighted Delaunay.

\begin{defn}\label{loc_del}
    Given a face $f_{ijk}$, we denote the distance between $O_{ijk}$ and edge $e_{ij}$ by $h_{ij,k}$, which is positive when $O_{ijk}$ and $v_k$ are on the same side of $e_{ij}$, and negative when they are on the different sides.
    
    An edge $e_{ij}$ is called \emph{local weighted Delaunay} if $h_{ij,k} + h_{ij,l} \ge 0$.
\end{defn}

Note that $h_{ij,k} + h_{ij,l} \ge 0$ implies the existence of $e_{ij}$.
Given a hinge $\Diamond_{ij; kl}$, if the two orthogonal circles of faces $f_{ijk}$ and $f_{ijl}$ do not intersect or intersect with the angle no more than $\frac{\pi}{2}$, then the edge $e_{ij}$ is local weighted Delaunay. If $h_{ij,k} + h_{ij,l} = 0$, we delete the edge $e_{ij}$ and insert the edge $e_{kl}$, then $e_{kl}$ is local weighted Delaunay. This operation is called \emph{diagonal switch}.

From the lemmas and definitions above, we derive main theorem of this section.

\begin{thm}\label{local_whole_f}
    Given a  piecewise flat surface $(S,V,d_f)$ with a weight function $\mathbf{w} \in W$, we know that:
    \begin{itemize}
        \item The weighted Delaunay triangulation exists.
        \item The triangulation is weighted Delaunay, if and only if all the edges are local weighted Delaunay \cite{glickenstein2008geometric}\cite{bobenko2007discrete}\cite{gorlina2011weighted}.
        \item The weighted Delaunay triangulation is unique up to finite diagonal switches.
    \end{itemize}
\end{thm}

The construction of definition \ref{del_def} implies the existence. Note that when $\mathbf{w} \notin W$ the weighted Delaunay triangulation may not exist, which is discussed in detail in \cite{gorlina2011weighted}.

Moreover, we have to discuss a special case only occurs on $\Delta$-complex triangulation, which is that two edges of one face glue together. No matter how they glue, forward or reverse, the face would be an isosceles triangle and the weights of two base angles are the same. Thus, the center of the orthogonal circle $O_{ijj}$ locates at the axis of symmetry, then $2h_{ij,j}>0$ and the self-glued edge is weighted Delaunay.

\section{Hyperbolic polyhedra surface}

\begin{lem}\label{cosh_ratio_cosh}
    Let $p_1,p_2 \in \mathbb{H}^2$ and $r_1,r_2 \in (0,d(p_1,p_2))$, then the set of all points $q \in \mathbb{H}^2$ satisfied
    \begin{equation}\label{vor_cosh}
        \frac{\cosh d(q,p_1)}{\cosh r_1}=\frac{\cosh d(q,p_2)}{\cosh r_2}
    \end{equation}
    is a geodesic, which is orthogonal to the geodesic segment connecting $p_1,p_2$. If the equal sign in the equation\eqref{vor_cosh} is changed to the less-than sign, this will be the half space containing $p_1$ partitioned by the geodesic, and if it is changed to the greater-than sign, this will be the half space containing $p_2$.
\end{lem}	

\begin{proof}
  
    Let the geodesic connecting $p_1,p_2$ be $\gamma$, define
    \begin{equation}
        h \colon \mathbb{H}^2 \to \mathbb{R} \quad
        p \mapsto \frac{\cosh d(p,p_1)}{\cosh r_1}-\frac{\cosh d(p,p_2)}{\cosh r_2}.
    \end{equation}

    As $0<r_1,r_2<d(p_1,p_2)$, we know $h(p_1)<0<h(p_2)$. From the continuity of the distance function, there exists $p_0 \in \gamma$ such that $h(p_0)=0$. Denote the geodesic passing through $p_0$ and perpendicular to $\gamma$ by $\mu$, for any $q \in \mu$, from the hyperbolic cosine triangle formula, we have
    \begin{equation}
        \begin{aligned}
            \cosh d(q,p_1)&=\cosh d(p_0,p_1)\cosh d(q,p_0),\\
            \cosh d(q,p_2)&=\cosh d(p_0,p_2)\cosh d(q,p_0).
        \end{aligned}
    \end{equation}
    Thus, $h(q)=0$.
    
    Now we show that $h$ has no other zero except $\mu$. 
    Let $\tilde{\gamma}$ be the extension geodesic of $\gamma$, suppose $q'$ is the orthogonal projection of a point $p'$ on $\tilde{\gamma}$, then $h(q')=h(p')\cosh d(q',p')$. Therefore, we need to prove that $p_0$ is the only zero on $\tilde{\gamma}$. Due to symmetry, it is only necessary to show that $h(p')<0$ when $p'$ and $p_1$ are on the same side of $p_0$. Compute the derivative
    \begin{equation}
        \frac{d}{dx}\frac{\cosh(d(p_0,p_1)-x)}{\cosh(d(p_0,p_2)+x)}
        =-\frac{\sinh(d(p_0,p_1)+d(p_0,p_2))}{\cosh^2(d(p_0,p_2)+x)}<0.
    \end{equation}    
    When $x>0$,
    \begin{equation}
        \frac{\cosh(d(p_0,p_1)-x)}{\cosh(d(p_0,p_2)+x)}<
        \frac{\cosh d(p_0,p_1)}{\cosh d(p_0,p_2)}=
        \frac{\cosh r_1}{\cosh r_2}
    \end{equation}
    where $x=d(p',p_0)$. Then we prove the equality, and the inequalities become trivial.
\end{proof}

Similar to the case of piecewise flat surfaces, we define the weight function on piecewise hyperbolic surface $(S,V,d_h)$, and get the weighted Voronoi decomposition. From now on, instead of using the square of the radius, we use the radius as the weight directly. Define the radii as $\mathbf{r} \colon V \to \mathbb{R}_{>0},\, v_i \mapsto r_i$ or $\mathbf{r} \in \mathbb{R}_{>0}^V$, the domain of definition $W,W'$ is revised to
\begin{equation}
    \begin{aligned}
        R& = \left\{\,\mathbf{r} \in \mathbb{R}_{>0}^V \mid 
        r_i<\sqrt{d_f(v_i,v_j)^2+r_j^2},\forall j \ne i \,\right\}
            \\
        R'& = \left\{\,\mathbf{r} \in \mathbb{R}_{>0}^V \mid
        r_i<\mathrm{Inj}(v_i) \mbox{ and }
        \forall i \ne j,r_i+r_j<d_f(v_i,v_j) \,\right\},
    \end{aligned}
\end{equation}
and the definition of inner Voronoi cell should be replaced correspondingly. For the case of piecewise hyperbolic surface, strictly speaking, the domains of weight function should be
\begin{equation}
    \begin{aligned}
        R& = \left\{\,\mathbf{r} \in \mathbb{R}_{>0}^V \mid 
        r_i<\arcch\left(\cosh(r_j)d_h(v_i,v_j)\right) ,\forall j \ne i \,\right\}
            \\
        R'& = \left\{\,\mathbf{r} \in \mathbb{R}_{>0}^V \mid
        r_i<\mathrm{Inj}(v_i) \mbox{ and }
        \forall i \ne j,r_i+r_j<d_h(v_i,v_j) \,\right\},
    \end{aligned}
\end{equation}
The only difference between the two cases is the subscript of distance, so we use the same symbol and differ them from context.

In this section, We aim to prove that the theorems for piecewise flat surface holds for piecewise hyperbolic surface as well.

\begin{defn}
    Given a piecewise hyperbolic surface $(S,V,d_h)$ and a weight function $\mathbf{r} \in R$, let $p \in S$, if there exists a unique geodesic connecting $p$ and $v_i$ whose length is equal to the distance $d_h(p,v_i)$, and for any $j \ne i$,
    \begin{equation}
        \frac{\cosh d_h(p,v_i)}{\cosh r_i}<\frac{\cosh d_h(p,v_j)}{\cosh r_j}
    \end{equation}
    is satisfied
    holds, then the set of point $p$ satisfying the conditions is called \emph{inner Voronoi cell (of piecewise hyperbolic surface case)} with respect to the vertex $v_i$, denoted by $Vor_h(v_i)$.
\end{defn}

\begin{thm}\label{vor_h}		
    There exists a unique CW decomposition of $S$, called \emph{weighted Voronoi decomposition}, whose 2-cells are $Vor_h(v_i)$ where $v_i \in V$.
\end{thm}

The proof of the theorem above follows the same method with that of theorem \ref{vor_f}, but some steps are slightly different, so we revise it as follows:

\begin{itemize}
    \item In Proposition \ref{vertex_vor_f}, for any $j \ne i$
    \begin{equation}
        \frac{\cosh d_h(v_i,v_i)}{\cosh r_i}<
        \frac{\cosh d_h(v_i,v_j)}{\cosh r_j},
    \end{equation}
    and 
    \begin{equation}
        \frac{\cosh d_h(v_j,v_i)}{\cosh r_i}>
        \frac{1}{\cosh r_j}=\frac{\cosh d_h(v_j,v_j)}{\cosh r_j}.
    \end{equation}

    \item In Proposition \ref{neighbor}, when we prove that the folded geodesic segment should not pass through vertices, we should calculate like
    \begin{equation}
        \frac{\cosh d_h(p,v_j)}{\cosh r_j}<\cosh d_h(p,v_j)<
        \frac{\cosh d_h(p,v_j)}{\cosh d_h(v_i,v_j)}<
        \frac{\cosh d_h(p,v_i)}{\cosh r_i}.
    \end{equation}

    In Proposition \ref{simply connected}, when we use Gauss-Bonnet theorem, the sum of the discrete curvatures of all vertices in a simply connected region $U$ should be $\alpha+\beta+Area(U)>0$
\end{itemize}

So far, we have proved the theorem on the existence and uniqueness of weighted Voronoi decomposition on piecewise hyperbolic surface. Then we consider its dual cell decomposition.

\begin{lem}\label{dual_isotopy2}
    Given a weighted Voronoi decomposition of $(S,V,d_h)$ with weight $\mathbf{r} \in R$, let $p$ be a point on a 1-cell which is the boundary of $Vor_h(v_i)$ and $Vor_h(v_j)$. Here, $v_i$ and $v_j$ could coincide. Suppose the path $\gamma$ is the geodesic lines connecting $p,v_i$ and $p,v_j$ glued at $p$, then there exists a unique geodesic on $S$ connecting $v_i$ and $v_j$ isotopic to $\gamma$ fixing $V$.
\end{lem}

This lemma is the hyperbolic version of Lemma \ref{dual_isotopy}, whose arguments hold when replacing $Vor_f$ with $Vor_h$. It is worth noting that Lemma \ref{dual_isotopy} does not hold for spherical case because the angle $\bigl<v_i,\mu(s'),\mu(1)\bigr>$ being no less than $\frac{\pi}{2}$ is not contradicted to that $\mu(1)$ having neighborhood intersect with $Vor_s(v_i)$, where $s$ represents \emph{sphere}.

Based on the previous work, the Definition \ref{del_def} and Theorem \ref{local_whole_f} can be copied directly, without any modification.

\begin{defn}\label{del_def2}
    Given a piecewise hyperbolic surface $(S,V,d_h)$ with weight $\mathbf{r} \in R$, there exist a unique CW decomposition whose 1-cells are geodesics connecting vertices, called \emph{weighted Delaunay tessellation}, to be the dual graph of the weighted Voronoi decomposition.
    
    The \emph{weighted Delaunay triangulation} is to subdivide polygon faces of the weighted Delaunay tessellation into triangles by connecting some geodesic diagonals in any ways.
\end{defn}

Similar to the Definition \ref{loc_del} of $h_{ij,k}$ in the Euclidean case, we define the edge $e_{ij}$ as \emph{local weighted Delaunay} to be $h_{ij,k}+h_{ij,l} \ge 0$ in the hinge $\Diamond_{ij;kl}$.

When $\mathbf{r} \in R'$, for any face $f_{ijk} \in F(\mathcal{T})$, there exists a unique geodesic circle or centered at the dual point of the surface $f_{ijk}$, or $O_{ijk}$, with the radius $\arcch(\frac{\cosh d_h(O_{ijk},v_\alpha)}{\cosh r_\alpha})$, orthogonal to the geodesic circles centered at $v_\alpha$ with radius $r_\alpha$, where $\alpha=i,j,k$. Similarly, we call such a circle the orthogonal circle here. It is immersed into $S$, and does not intersect or intersects but with an intersection angle not greater than $\frac{\pi}{2}$ with any other circle centered at vertices.

\begin{thm}\label{local_whole_h}
    Given a piecewise hyperbolic surface $(S,V,d_h)$ with weight $\mathbf{r} \in R$,
    \begin{itemize}
        \item The weighted Delaunay triangulations exist.
        \item A triangulation is weighted Delaunay, if and only if all the edges are local weighted Delaunay.
        \item The weighted Delaunay triangulation is unique up to diagonal switch.
    \end{itemize}
\end{thm}	
The first and third items are obviously valid through an analogy with theorem \ref{local_whole_f}.
Although there is no reference directly proving the second item, the methods described in the articles \cite{bobenko2007discrete} and \cite{gorlina2011weighted} can be directly extended to this hyperbolic case.

\section{Hyperbolic surface with geodesic boundary}

The lemma below can be found in \cite{buser2010geometry}, which claims that it is easy to prove using cyclic permutations. However, no detailed procedure for the special case following is provided. For completeness, we derive the formula in the lemma by a direct computation.

\begin{lem}\label{cosh_ratio}
    Let $ABCD$ be a hyperbolic quadrilateral and the angles $A,B,C$ be $\frac{\pi}{2}$, then,
    \begin{equation}
        \cosh AB=\frac{\tanh AD}{\tanh BC} \quad
        \cosh CD=\frac{\sinh AD}{\sinh BC}.
    \end{equation}
\end{lem}

\begin{figure}[ht]
    \centering \includegraphics{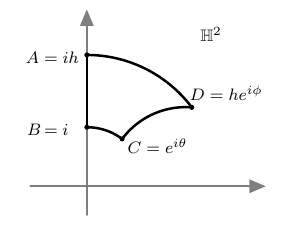}
    \caption{Coordinates of hyperbolic quadrilateral with three right angles} \label{fig:hyp31}
\end{figure}

\begin{proof}		
    With an isometric embedding from the hyperbolic quadrilateral into the hyperbolic space of the upper half plane, we map the points to the coordinates $A=i h, B=i, C=e^{i\theta}, D=h e^{i\phi}$.The set of equations of the geodesics $CD,AD$ is
    \begin{equation}
        \begin{aligned}
            \left( x - \frac{1}{\cos \theta } \right)^2 + y^2 &= \tan ^2\theta\\
            x^2+y^2&=h^2,
        \end{aligned}			
    \end{equation}
    whose solution is $x=\frac{(h^2+1)}{2} \cos \theta$, thus we have $\cos \phi = \frac{(h^2+1)}{2 h} \cos \theta$.

    By the geodesic length formula of the hyperbolic upper half plane, we know that
    \begin{equation}
        AD=\log h \quad BC=\log \cot \frac\theta2 \quad AD=\log \cot \frac\phi2.
    \end{equation}
    By substituting the above equations, we have $\cosh AB\tanh BC=\tanh AD$. In the same way, $\tanh CD =\tanh AB \cosh BC$. Therefore,
    \begin{equation}			
        \begin{aligned}
            &\cosh^{-2}CD=1-\tanh^2 CD\\
            =&1-\tanh^2 AB \cosh^2 BC
            =1-\frac{\cosh^2 AB-1}{\cosh^2 AB}\cosh^2 BC\\
            =&1-\frac{\tanh^2 AD-\tanh^2 BC}{\tanh^2 AD}\cosh^2 BC\\
            =&\frac{\tanh^2 AD(1-\cosh^2 BC)-\sinh^2 BC}{\tanh^2 AD}\\
            =&\frac{(\tanh^2 AD-1)\sinh^2 BC}{\tanh^2 AD}
            =\frac{\sinh^2 BC}{\sinh^2 AD}.\\
        \end{aligned}
    \end{equation}    
    Since the length is positive, we take the square root of both sides, then
    \[\cosh CD\sinh BC=\sinh AD.\]
\end{proof}

In the metric space $(M,d)$, consider two subsets $A,B\subset M$. We denote $d(A,B)\coloneqq \inf_{p\in A,q\in B}d(p,q)$ as the distance between the two sets. 
Let $\gamma$ be the geodesic in hyperbolic plane $\mathbb{H}^2$, for any $p \in \mathbb{H}^2$, there is a unique $p_0 \in \gamma$ such that $d(p,p_0)=d(p,\gamma)$ and the geodesic line segment between $p,p_0$ is orthogonal to $\gamma$.
Let $\gamma_1,\gamma_2$ be two geodesics that do not intersect in $\mathbb{H}^2$, there exists a unique $p_1 \in \gamma_1$ and a unique $p_2 \in \gamma_2$ such that $d(p_1,p_2)=d(\gamma_1,\gamma_2)$ and the geodesic line segment between $p_1,p_2$ is orthogonal to $\gamma_1,\gamma_2$.

By Lemma \ref{cosh_ratio}, using the proof method of Lemma \ref{cosh_ratio_cosh}, we can get the following lemma.

\begin{lem}\label{gbh_sinh}
    Let $r_1,r_2>0$ and $\gamma_1,\gamma_2$ be two geodesics that do not intersect in $\mathbb{H}^2$, then the set of all points $q \in \mathbb{H}^2$ satisfied
    \begin{equation}\label{vor_sinh}
        r_1 \sinh d(q,\gamma_1)=r_2 \sinh d(q,\gamma_2),
    \end{equation}
    is a geodesic, which is orthogonal to the geodesic segment connecting $\gamma_1,\gamma_2$. If the equal sign in the equation \eqref{vor_sinh} is changed to the less-than sign, this will be the half space containing $\gamma_1$ partitioned by the geodesic, and if it is changed to the greater-than sign, this will be the half space containing $\gamma_2$.
\end{lem}

\begin{defn}
    Given a hyperbolic surface $(\Sigma,d)$ with geodesic boundaries $\partial \Sigma=\{\gamma_i,\dots,\gamma_n\}$ and a weight function $\mathbf{r} \in \mathbb{R}^V_{>0}$, 
    the \emph{inner weighted Voronoi cell} of $v_i$ in this case, denoted by $Vor(\gamma_i)$, is defined to be the set of all $p \in S$ such that there exists a unique geodesic connecting $p$ and $\gamma_i$ whose length is the distance $d(p,\gamma_i)$, and for any $j \ne i$, the inequality
    \begin{equation}
        r_i \sinh d(v_i,\gamma_i)<r_j \sinh d(v_j,\gamma_j)
    \end{equation}
    holds.
\end{defn}

\begin{thm}\label{gbh_vor}
    Let $(S,V)$ be the related surface of $\Sigma$, namely, $\Sigma = S \setminus \bigcup_{i=1}^n D_i$ where $D_i$ are disjoint open disks containing $v_i$. Given a weight function $\mathbf{r} \in \mathbb{R}^V_{>0}$, there exists a CW decomposition of $S$, with $\{\, Vor(\gamma_i) \cup D_i \mid i=1,\dots,n \,\}$ as 2-cells. This decomposition is unique restricted on $\Sigma$.
\end{thm}

We just ignore the metric on $D_i$. The only reason for gluing $D_i$ on $\Sigma$ is to make $Vor(\gamma_i)$ simply connected such that it can be call a cell decomposition. If not, since $Vor(\gamma_i)$ is doubly connected, the disjoint union should be called \emph{division} instead of cell decomposition. However, we continue to use the previous name, which is called \emph{weighted Voronoi decomposition} of $\Sigma$ or $(S,V)$ for convenience.

This theorem can also copy the proof of theorem \ref{vor_f} with some modifications. Basically, we change the vertex into the boundary, and the isotopy fixing vertices into normal isotopy. Here are some steps that have changed greatly in the modification, in which we define $\gamma_i=\partial D_i$.

\begin{prop}\label{gbh_subset}
    $\gamma_i \subset Vor(\gamma_i)$ and for any $j \ne i$, $\gamma_j \nsubset Vor_f(\gamma_i)$.
\end{prop}

\begin{proof}
    For any $j \ne i$, since $\gamma_i \cap \gamma_j=\varnothing$, when $p \in \gamma_i$,  we have $0=d(p,\gamma_i)<\frac{r_j}{r_i}d(p,\gamma_j)$. Therefore, $\gamma_i \subset Vor(\gamma_i)$ and $\gamma_i \nsubset Vor(\gamma_j)$.
\end{proof}

\begin{prop}\label{open2}		
    $Vor(\gamma_i) \cup D_i$ is an open set in $S$.
\end{prop}

\begin{prop}\label{closure2}	
    $\bigcup_{v_i \in V} \overline{Vor(\gamma_i)} = \Sigma$ and		
    \begin{equation}
        \overline{Vor(\gamma_i)}=\{\,p \in \Sigma \mid
        r_i \sinh d(v_i,\gamma_i) \le r_j \sinh d(v_j,\gamma_j),
        \forall j \ne i\,\}
    \end{equation}
\end{prop}

\begin{prop}\label{neighbor2}		
    For any $\forall p \in \overline{Vor(\gamma_i)}$, let $\gamma \colon I \to S$ be the shortest geodesic line connecting $p$ and $\gamma_i$, and $\gamma(0) \in \gamma_i,\gamma(1)=p$, then there exists the neighborhood $U$ of $\gamma$ (in the subspace topological sense of $\Sigma \subset S$) isometrically embedding in $\mathbb{H}^2$.
\end{prop}	 

When proving this proposition, it is not necessary to exclude the case when $\gamma$ passes through other $\gamma_j$, because it is impossible for this case to happen on the surface with constantly negative Gauss curvature by the variation. The rest of the proof is trivial.

\begin{prop}	
    $Vor(\gamma_i) \cup D_i$ is simply connected.
\end{prop}	

We can prove it by the uniqueness of geodesic, without using Gauss-Bonnet theorem.

After these modifications, we have proved theorem \ref{gbh_vor}.

For the Voronoi-Delaunay duality, we need the following lemma.
\begin{lem}\label{gbh_dual_isotopy}
    Let $p$ belong to a one-dimensional cell of a weighted Voronoi decomposition of a hyperbolic surface with geodesic boundaries. This cell is the boundary of $Vor(\gamma_i),Vor(\gamma_j)$ (here, the neighborhood of $p$ is allowed to be the same 2-cell, or $\gamma_i=\gamma_j$). If one connect the geodesic connecting $p,\gamma_i$ and geodesic connecting $p,\gamma_j$ at $p$ to obtain the path $\gamma$, then there exists a unique geodesic connecting $\gamma_i,\gamma_j$ isotopic to $\gamma$ in $\Sigma$. 
\end{lem}

Since there is no cone point in $\Sigma$, there is no need for further discussion. We know that the lemma holds by the variational principle of the geodesic.

Let $(S,V)$ be the marked surface associated to $\Sigma$ with the weight $\mathbf{r} \in \mathbb{R}^V_{>0}$, then their exists a cell decomposition dual to its weighted Voronoi decomposition, called the \emph{weighted Delaunay tesselation}, which is unique limited on $\Sigma$ up to isometry isotopic to identity. Each cell is a hyperbolic right-angled polygon with even edges, of which the alternate edges are contained in $\partial\Sigma$. By connecting the geodesics between boundaries of a polygon with no less than eight edges, we get a cell decomposition with hyperbolic right-angled hexagons, or \emph{truncated triangulation} precisely. We also call the different ways for choosing diagonal geodesics of polygon \emph{switches}.

Similar to the method of proving theorem \ref{local_whole_f} by Lemma \ref{dual_isotopy}, the following theorem can be proved by Lemma \ref{gbh_dual_isotopy}.

\begin{thm}\label{gbh_unique}
    Given a hyperbolic surface with geodesic boundaries $(\Sigma,d)$ and a weight function $\mathbf{r} \in \mathbb{R}^V_{>0}$, there exists a weighted Delaunay triangulation, which is unique up to finite edge switches.
\end{thm}

\section{Finiteness}

\begin{lem}\label{finite_edges}
    
    Given any two points $p,q$ on a surface $(S,V,d_f)$ or $(S,V,d_h)$ or $(\Sigma,d)$ ($p$ and $q$ can coincide), then there is a finite number of geodesics with bounded lengths connecting $p$ and $q$ without passing other vertices. In particular, for $(\Sigma,d)$, there is a finite number of geodesics with bounded lengths that are connected and orthogonal to two given boundaries.    
    
\end{lem}

The proof of the lemma above used the Arzel{\`a}–Ascoli theorem in the case of piecewise flat surfaces \cite{indermitte2001voronoi}, and this method can be directly generalized to the other two kinds of surface.

\begin{thm}
    
    Given a piecewise flat surface $(S,V,d_f)$, we can find a finite number of triangulations $\mathcal{T}_1,\dots,\mathcal{T}_k$, such that for any weight function $\mathbf{r} \in R$, there exist a weighted Delaunay triangulation $\mathcal{T}_i$ among the finite collection.    

\end{thm}

\begin{proof} 

    Since $(S,V,d_f)$ is compact, the diameter of $S$ denoted by $D=\sup\limits_{p,q \in S}d_f(p,q)$ is bounded. Given any weight $\mathbf{r} \in R_1$, by Lemma \ref{dual_isotopy}, any edge of the corresponding weighted Delaunay triangulation is isotopic to a folded line segment glued by two geodesics. These two geodesics respectively connect a point to two vertices, or to one vertex in different isotopy class. Thus, the edge of weighted Delaunay triangulation is no longer than the length of the folded line segment, so it is no longer than $2D$.
    
    According to Lemma \ref{finite_edges}, we know that the number of edges connecting the given pair of vertices with the length less than $2D$ is finite.
    Note that the number of ways for choosing vertex pairs is no more than $m\left|V\right|^2$, where $m$ represents the maximum number of the 1-cell of weighted Voronoi decomposition.     
    Therefore, the edges of weighted Delaunay triangulations are selected from a finite number of possible edges. As this number is only related to $(S,V,d_f)$, there is only a finite number of ways to select such edges, thus the number of triangulations is finite.

\end{proof}

    It should be pointed out that in the thesis \cite{gorlina2011weighted}, to demonstrate the termination of edge flipping, a similar method was used to prove the finiteness. Specifically, when the weights are fixed and Delaunay condition is not required, but edge lengths are bounded, then the number of triangulations is finite. This differs from the conditions in our theorem where weights are allowed to vary and weighted Delaunay condition should be required.

    We can prove the following two theorems with the same method.

\begin{thm}

    Given a piecewise hyperbolic surface $(S,V,d_h)$, we can find a finite number of triangulations $\mathcal{T}_1,\dots,\mathcal{T}_k$, such that for any weight function $\mathbf{r} \in R$, there exist a weighted Delaunay triangulation $\mathcal{T}_i$ among the finite collection.    

\end{thm}

\begin{thm}\label{akiyoshi}
    
    Given a hyperbolic surface with geodesic boundaries $(\Sigma,d)$, we can find a finite number of truncated triangulations $\mathcal{T}_1,\dots,\mathcal{T}_k$, such that for any weight function $\mathbf{r} \in \mathbb{R}^n_{>0}$, there exist a weighted Delaunay triangulation $\mathcal{T}_i$ among the finite collection, where $n$ is the number of boundaries.    

\end{thm}

\bibliographystyle{alpha}
\bibliography{ref}

\end{document}